\documentclass[nointlimits]{amsart}

\usepackage[T1]{fontenc} 
\usepackage[utf8]{inputenc} 
\usepackage[USenglish]{babel} 

\usepackage{lmodern}
\usepackage{mathrsfs}
\usepackage{microtype}

\usepackage{url}\usepackage{natbib}

\usepackage{amsmath, amssymb, amsthm, enumerate}

\usepackage[normalem]{ulem}

\usepackage{enumerate}

\usepackage{hyperref}
\newcommand{\TITLE}{Non-improvability of sharp endpoint estimates}
\hypersetup{
	unicode=true,
	breaklinks=true,
	pdftitle={\TITLE},
	pdfauthor={Zdeněk Mihula, Luboš Pick, and Armin Schikorra}
	}
	
\makeatletter
\let\save@mathaccent\mathaccent
\newcommand*\if@single[3]{%
  \setbox0\hbox{${\mathaccent"0362{#1}}^H$}%
  \setbox2\hbox{${\mathaccent"0362{\kern0pt#1}}^H$}%
  \ifdim\ht0=\ht2 #3\else #2\fi
  }
\newcommand*\rel@kern[1]{\kern#1\dimexpr\macc@kerna}
\newcommand*\widebar[1]{\@ifnextchar^{{\wide@bar{#1}{0}}}{\wide@bar{#1}{1}}}
\newcommand*\wide@bar[2]{\if@single{#1}{\wide@bar@{#1}{#2}{1}}{\wide@bar@{#1}{#2}{2}}}
\newcommand*\wide@bar@[3]{%
  \begingroup
  \def\mathaccent##1##2{%
    \let\mathaccent\save@mathaccent
    \if#32 \let\macc@nucleus\first@char \fi
    \setbox\z@\hbox{$\macc@style{\macc@nucleus}_{}$}%
    \setbox\tw@\hbox{$\macc@style{\macc@nucleus}{}_{}$}%
    \dimen@\wd\tw@
    \advance\dimen@-\wd\z@
    \divide\dimen@ 3
    \@tempdima\wd\tw@
    \advance\@tempdima-\scriptspace
    \divide\@tempdima 10
    \advance\dimen@-\@tempdima
    \ifdim\dimen@>\z@ \dimen@0pt\fi
    \rel@kern{0.6}\kern-\dimen@
    \if#31
      \overline{\rel@kern{-0.6}\kern\dimen@\macc@nucleus\rel@kern{0.4}\kern\dimen@}%
      \advance\dimen@0.4\dimexpr\macc@kerna
      \let\final@kern#2%
      \ifdim\dimen@<\z@ \let\final@kern1\fi
      \if\final@kern1 \kern-\dimen@\fi
    \else
      \overline{\rel@kern{-0.6}\kern\dimen@#1}%
    \fi
  }%
  \macc@depth\@ne
  \let\math@bgroup\@empty \let\math@egroup\macc@set@skewchar
  \mathsurround\z@ \frozen@everymath{\mathgroup\macc@group\relax}%
  \macc@set@skewchar\relax
  \let\mathaccentV\macc@nested@a
  \if#31
    \macc@nested@a\relax111{#1}%
  \else
    \def\gobble@till@marker##1\endmarker{}%
    \futurelet\first@char\gobble@till@marker#1\endmarker
    \ifcat\noexpand\first@char A\else
      \def\first@char{}%
    \fi
    \macc@nested@a\relax111{\first@char}%
  \fi
  \endgroup
}
\makeatother

\numberwithin{equation}{section}

\theoremstyle{plain}
\newtheorem{theorem}{Theorem}[section]
\newtheorem{corollary}[theorem]{Corollary}

\newtheorem{example}[theorem]{Example}

\theoremstyle{definition}
\newtheorem{definition}[theorem]{Definition}
\newtheorem{remark}[theorem]{Remark}
\newtheorem{convention}[theorem]{Convention}

\newcommand{\R}{\mathbb{R}}
\newcommand{\rn}{\R^n}
\newcommand{\HH}{\mathcal{H}}
\newcommand{\N}{\mathbb{N}}
\newcommand{\M}{\mathscr{M}}

\newcommand{\RR}{\mathcal{R}}
\newcommand{\SSS}{\mathcal{S}}
\newcommand{\RM}{(\RR,\mu)}
\newcommand{\SN}{(\SSS,\nu)}
\newcommand{\MOR}{\M_0\RM}
\newcommand{\MOS}{\M_0\SN}

\newcommand{\C}{\mathscr{C}}

\newcommand*\dd{\mathop{}\!\mathrm{d}}

\title{\TITLE}
\author{Zden\v ek Mihula}
\address{Zden\v ek Mihula, Czech Technical University in Prague, Faculty of Electrical Engineering, Department of Mathematics, Technick\'a~2, 166~27 Praha~6, Czech Republic}
\email{mihulzde@fel.cvut.cz}
\urladdr{\href{https://orcid.org/0000-0001-6962-7635}{0000-0001-6962-7635}}

\author{Lubo\v{s} Pick}
\address{Lubo\v{s} Pick, Department of Mathematical Analysis,
Faculty of Mathematics and Physics,
Charles University,
So\-ko\-lo\-vsk\'a~83,
186~75 Praha~8,
Czech Republic}
\email{pick@karlin.mff.cuni.cz}
\urladdr{\href{https://orcid.org/0000-0002-3584-1454}{0000-0002-3584-1454}}

\author{Armin Schikorra}
\address{Armin Schikorra, Department of Mathematics,
University of Pittsburgh,
301 Thackeray Ave,
Pittsburgh, PA 15260, United States of America}
\email{armin@pitt.edu}
\urladdr{\href{https://orcid.org/0000-0001-9242-1782}{0000-0001-9242-1782}}

\begin{document}
\setcitestyle{numbers}
\bibliographystyle{plainnat}

\thanks{This research was partly supported by grant no.~23-04720S of the Czech Science Foundation and grant no.~DMS-2044898 (CAREER) of the US National Science Foundation. A.S. is a Humboldt Fellow.}

\begin{abstract}
For an integer $n$ and the parameter $\gamma\in(0,n)$, the Riesz potential $I_\gamma$ is known to take boundedly $L^1(\rn)$ into $L^{\frac{n}{n-\gamma},\infty}(\rn)$, and also that the target space is the smallest possible among all rearrangement-invariant Banach function spaces. We study the natural question whether the target space can be improved when the domain space is replaced with a (smaller) Lorentz space $L^{1,q}(\rn)$ with $q\in(0,1)$. The classical methods cannot be used because the spaces $L^{1,q}(\rn)$ are not equivalently normable. We develop two new abstract methods, establishing rather general results, a particular consequence of each (albeit achieved through completely different means) being the negative answer to this question. The methods are based on special functional properties of endpoint spaces. The results can be applied to a wide field of operators satisfying certain minimal requirements.
\end{abstract}

\maketitle


\section{Introduction and general overview}
The \emph{Riesz potential} is defined, for $n\in\N$ and $\gamma\in(0,n)$, by the formula
\begin{equation*}
    I_{\gamma}f(x) =
    c_{\gamma}
    \int_{\rn}\frac{f(y)}{|x-y|^{n-\gamma}}\,dy
    \quad\text{for $x\in\rn$,}
\end{equation*}
where $c_{\gamma}$ is the normalization constant given by
\begin{equation*}
    c_{\gamma} =
    \frac{\pi^{\frac{n}{2}}2^{\gamma}\Gamma(\frac{\gamma}{2})}{\Gamma(\frac{n-\gamma}{2})}.
\end{equation*}
The Riesz potentials constitute a scale of linear operators and are known to be of great importance in various subdisciplines of analysis and its applications. One of the most important questions concerning these operators is how they act on function spaces.

It is known that, for $p\in(1,\frac{n}{\gamma})$, one has
\begin{equation}
\label{E:riesz-on-lebesgue}
    I_{\gamma}\colon L^{p}(\rn)\to L^{\frac{np}{n-\gamma p}}(\rn),
\end{equation}
meaning that the operator $I_{\gamma}$ takes boundedly the Lebesgue space $L^{p}(\rn)$ into another Lebesgue space, $L^{\frac{np}{n-\gamma p}}(\rn)$. This result is classical and sometimes is referred to as the \emph{Hardy--Littlewood--Sobolev theorem on Riesz potentials}, see e.g.~\cite[Theorem~5.5]{Sad:79}. It is also known, and can easily be verified using radially decreasing symmetric functions, that~\eqref{E:riesz-on-lebesgue} is no longer true when $p=1$. A similar discrepancy is met at the opposite endpoint, namely when $p=\frac{n}{\gamma}$. Then the expression $\frac{np}{n-\gamma p}$ can be considered infinity, but the resulting analogue of~\eqref{E:riesz-on-lebesgue}, that is,
\begin{equation*}
    I_{\gamma}\colon L^{\frac{n}{\gamma}}(\rn)\to L^{\infty}(\rn),
\end{equation*}
is not true. In order to get suitable replacements in either of these borderline cases, one has to resort to more general spaces than the Lebesgue ones. A particularly suitable scale of function spaces for this purpose is that of \emph{two-parameter Lorentz spaces}. 

Given a $\sigma$-finite nonatomic measure space $(\RR,\mu)$ and $p,q\in(0,\infty]$, the \emph{Lorentz functional} $\|\cdot\|_{p,q}$ is defined, for every $\mu$-measurable function $f\colon \RR\to\R$ by
\begin{equation*}
    \|f\|_{p,q}
    =
    \begin{cases}\displaystyle
				\left( p \int_{0}^{\infty}
					\bigl[s\, \mu(\{x\in\RR:|f(x)|>s\})^\frac{1}{p}\bigr]^q \,\frac{\dd s}{s}
				\right)^\frac{1}{q}
					& \text{if $q\in(0,\infty)$},
					\\[\bigskipamount]
				\sup\limits_{s\in(0,\infty)} s\, \mu(\{x\in\RR:|f(x)|>s\})^
                \frac{1}{p}
					& \text{if $q=\infty$.}
			\end{cases}
\end{equation*}
The value $\|f\|_{p,q}$ might be finite or infinite, and the corresponding \emph{Lorentz space} $L^{p,q}(\RR,\mu)$ is then defined as the collection of all those $f$ for which one has $\|f\|_{p,q}<\infty$. When $p=q$, the Lorentz space $L^{p,q}\RM$ coincides with the Lebesgue space $L^p\RM$, with equal (quasi-)norms (e.g., \cite[Theorem~1.13]{LL:01}). Lorentz and Lebesgue spaces are typical examples of (quasi-)normed rearrangement-invariant structures (function spaces), in which equimeasurable functions (i.e., having the same measure of upper level sets) have the same (quasi-)norm.

Equipped with the definition of Lorentz spaces, we can now enhance~\eqref{E:riesz-on-lebesgue} twofold. First, for $p\in(1,\frac{n}{\gamma})$, we get a sharper estimate 
\begin{equation}
\label{E:riesz-on-lorentz}
    I_{\gamma}\colon L^{p}(\rn)\to L^{\frac{np}{n-\gamma p},p}(\rn),
\end{equation}
and, second, we can formulate sensible replacements for both borderline estimates, namely 
\begin{equation}
\label{E:riesz-on-l-1}
    I_{\gamma}\colon L^{1}(\rn)\to L^{\frac{n}{n-\gamma},\infty}(\rn)
\end{equation}
and
\begin{equation}
\label{E:riesz-on-l-infty}
    I_{\gamma}\colon L^{\frac{n}{\gamma},1}(\rn)\to L^{\infty}(\rn).
\end{equation}
Two-parameter Lorentz spaces enjoy the following important scaling property: one has
\begin{equation}
\label{E:scaling-lorentz}
    L^{p,q}\hookrightarrow L^{p,r}
    \quad\text{whenever $p\in(0,\infty]$ and $0<q\le r\le\infty$,}
\end{equation}
where $\hookrightarrow$ stands for a continuous embedding. Remarkably,~\eqref{E:scaling-lorentz} holds regardless of the underlying measure space (in particular, it does not matter whether $\mu(\RR)$ is finite or not). Owing to the relation~\eqref{E:scaling-lorentz}, applied with the triple $(p,q,r)$ replaced with $(\frac{np}{n-\gamma p}, p,\frac{np}{n-\gamma p})$, and observing the trivial inequality $p<\frac{np}{n-\gamma p}$, we see that~\eqref{E:riesz-on-lorentz} is indeed an essential improvement of~\eqref{E:riesz-on-lebesgue}.  The relations~\eqref{E:riesz-on-lorentz}--\eqref{E:riesz-on-l-infty} are classical and can be found in many monographs (see e.g.~\cite{Z:89,AH:96}). An elementary proof of~\eqref{E:riesz-on-l-1} can be obtained, after some simple modifications, from~\cite[Lemma~2.1]{Mal:02}. Then,~\eqref{E:riesz-on-l-infty} follows by the self-adjoinness of the Riesz potential, and~\eqref{E:riesz-on-lorentz} as well as \eqref{E:riesz-on-lebesgue} by interpolating the former two.

A natural question, and, at the same time, one of the main motivations for the research to be seen below, is whether one can improve the target space in~\eqref{E:riesz-on-l-1} when the domain space $L^1(\rn)$ is replaced by the strictly smaller space $L^{1,q}(\rn)$ with $q\in(0,1)$ (note that here we call~\eqref{E:scaling-lorentz} into play once again). One of the principal consequences of the main results below is the \emph{negative} answer to this problem (hence the expression \emph{non-improvability} in the title). At the same time, we shall point out important features behind the scenes that govern such behavior of relatively general class of operators. 

The crucial ingredients will be the summing property of a Lorentz endpoint space and the extremal property of a certain privileged function in a Marcinkiewicz endpoint space. This deserves a more detailed explanation. 

Let us recall that one of the most important features of every (quasi-)normed rearrangement-invariant structure is its \emph{fundamental level}, determined by the behavior of the corresponding (quasi-)norm on  characteristic functions. This approach is enabled by the observation that, given a rearrangement-invariant functional $\|\cdot\|_X$ and $t$ in the range of the measure $\mu$, then one can define the value $\varphi_X(t)$ as $\|\chi_E\|_X$ regardless of the choice of $E$ as long as it is measurable and has measure equal to $t$. This $\varphi_X$ is then called the \emph{fundamental function} of $X$. The collection of all spaces sharing this parameter is called a fundamental level. 

On each fundamental level, one has two important structures with certain extremal properties, namely the \emph{Lorentz endpoint space} and the \emph{Marcinkiewicz enpoint space}. For instance, for a fixed $p\in(1,\infty)$, all the two-parameter Lorentz spaces $\{L^{p,q}\}$ with $q\in(0,\infty)$, belong to the same fundamental level. The corresponding endpoint spaces, which are the smallest and biggest normable spaces among them, are $L^{p,1}$ and $L^{p,\infty}$.

It has to be pointed out that the spaces $L^{1,q}$ with $q\in(0,1)$ require special care since they are not equivalently normable. This fact rules out use of some classical methods such as duality techniques which are often being applied when optimality of function spaces is in question, see e.g.~\cite{EKP:00,CPS:15,EMMP:20}. This makes the problem more interesting, and certainly substantially more difficult. 

One can, of course, also be interested in a problem which is in some sense `dual', namely to the question of optimality of the target space when the domain space is fixed. For the particular case of the Riesz potential, one would be interested whether there is any essential enhancement of the space $L^{\frac{n}{n-\gamma},\infty}(\rn)$ possible within the class of quasinormed spaces in~\eqref{E:riesz-on-l-1}. 

Finally, since the behavior of the Riesz potential is typical for many other important operators of harmonic analysis in far more general settings, it is worth to study the problem from a broader point of view. That is exactly our mission here. 

In this paper, we develop two new and rather abstract  methods for resolving questions about (non-)improvability of function spaces in sharp estimates of sublinear operators.
In a way, the methods tackle the problem from mutually opposite directions. In order to keep the theory at reasonable accessibility, we restrict ourselves to the situation when either the domain space is a Lorentz endpoint space or the target space is a Marcinkiewicz endpoint space (detailed definitions of these and other relevant notions will be given in the course of explanation below).

The first approach is based on a careful axiomatization and subsequent search of minimal conditions under which the optimal target space cannot be essentially enhanced despite substantial shrinking of the domain space along its fundamental level once the critical point has been crossed (in the case of the Riesz potential, the critical point is the space $L^1(\rn)$). 

We begin by introducing a new abstract notion of optimality of the target space within certain \emph{compatible triple} consisting of an assortment of quasi-Banach lattices, the operator in question and the fixed domain space. Our first main result (Theorem~\ref{thm:nonimprovability_Lorentz_endpoint}) states that under certain approximation-type requirements (which have to be axiomatized), the optimal partner target space of a Lorentz endpoint space remains optimal even when the domain space is replaced by any of the strictly smaller spaces having comparable fundamental functions. Since the situation described by~\eqref{E:riesz-on-l-1} is a particular case of that described by the statement of~Theorem~\ref{thm:nonimprovability_Lorentz_endpoint}, we get, as a consequence, the negative answer to the problem concerning the Riesz potential stated above (see~Example~\ref{exam:nonimprovability_of_riesz}).

Our second approach is based on adopting a completely different point of view, calling into play special properties of one extremal function in a Marcinkiewicz space. We introduce a 
pair of operators, one involving integration and the other using the operation of taking supremum. The introduction of these operators is not entirely new as it arises from the classical theory of joint weak-type operators, corroborated in~\cite{BR:80,BS}, see also~\cite{C:66}, and its various later modifications suitable for operators with nonstandard behavior, which was first developed in~\cite{Gog:09}, and then further studied e.g.~in~\cite{Mal:12,Bae:22,Mih:25,Kub:25}. Nevertheless, it seems that the idea of exploiting these, in a way extremal, operators to establish optimality within a large class of merely quasi-normable spaces is new. In the main results of this second part, namely Theorems~\ref{thm:nonimprovability_Calderon_endpoint_R} and~\ref{thm:nonimprovability_Calderon_endpoint_H}, we formulate other, completely different, sufficient conditions for the non-improvability of the sharp target space, which maintains its optimality even though significantly more function spaces compete for its place.

As a general framework for the study of problems we intend to pursue, we introduce a category of quasi-Banach lattices which are \emph{Fatou-representable} (see~Definition~\ref{def:Fatouably_repre_lattice}). It turns out that this class is extraordinarily satisfactory and enables us to obtain nonimprovability results in a surprisingly general setup. The example of the action of the Riesz potential on $L^1(\rn)$ mentioned above is, once again, a particular example, marking a common point of the two approaches.

As a final remark, let us point out that while we show here that \eqref{E:riesz-on-l-1} does not improve in the Lorentz space scale when $I_\gamma$ acts on $L^{1,q}$-functions, $q<1$, it is known that the Riesz potential acting on $L^1$-vector fields with additional \emph{cancelling} conditions, such as curl-freeness, does map into finer Lorentz spaces such as $L^{\frac{n}{n-\gamma},1}$, cf.    \cite{SSVS17,Spector20,Stolyarov22,HRS23}.

\section{Results and proofs}

Throughout this text, $\RM$ and $\SN$ stand for $\sigma$-finite nonatomic measure spaces. Unless necessary, no further specification will be given at each occurrence. Where distinct domain and target function spaces for an operator are involved, we use $\RM$ and $\SN$ on the domain and target sides, respectively.
Given such a measure space $\RM$, we will denote by $\M\RM$ the collection of all $\mu$-measurable functions $f\colon\RR\to[-\infty,\infty]$. By $\MOR$
we shall denote the subset of $\M\RM$ consisting of those functions which are finite $\mu$-a.e.~on $\RR$.

For a functional $\|\cdot\|_{X}\colon\M\RM\to[0,\infty]$, we denote by $X\RM$ (or $X$ for short) the set of all $f\in\M\RM$ such that $\|f\|_{X}<\infty$. We say that $X$ is a \emph{(quasi-)Banach lattice} if $X\subseteq \MOR$, $\|\cdot\|_X$ is a (quasi-)norm on $X$, it is complete with respect to the (quasi-)metric $\|f-g\|_{X}$, and the \emph{lattice property} holds in the sense that if $f\in\M\RM$, $g\in X$ and $|f|\le |g|$ $\mu$-a.e., then $f\in X$ and $\|f\|_X\le\|g\|_X$. 

We say that a collection of (quasi-)Banach lattices $\C$ is \emph{compatible} if there is a $\sigma$-finite nonatomic measure space $\RM$ such that each $X\in\C$ is contained in $\MOR$.

If $E\subseteq\RR$ is a $\mu$-measurable set, we denote the characteristic function of $E$ by $\chi_E$.

\begin{convention}\label{conv:qnorm_extension}
Given a (quasi-)Banach lattice $X$ over $\RM$, its (quasi-)norm $\|\cdot\|_X$ is a functional defined on $X$ with values in $[0, \infty)$. While $\|f\|_X$ may or may not make formal sense for $f\in \M\RM\setminus X$, from the point of view of $X$, $\|\cdot\|_X$ is defined only for $f\in X$. In this paper, we implicitly assume that $\|\cdot\|_X$ is extended (or redefined) to be defined for all $f\in\M\RM$ and set
\begin{equation*}
    \|f\|_X = \infty \quad \text{for $f\in \M\RM\setminus X$}.
\end{equation*}
With this, for every $f\in\M\RM$, we have
\begin{equation*}
    f\in X \quad \text{if and only if} \quad \|f\|_X < \infty.
\end{equation*}
\end{convention}

Note that in accordance with Convention~\ref{conv:qnorm_extension}, we handle subspaces of (quasi-)normed linear spaces in a slightly modified way than usual. For example, we treat $\|\chi_{\RR}\|_{L^\infty_0\RM}$ as $\infty$ when $\mu(\RR) = \infty$. By $L^\infty_0\RM$, we denote the subspace of $L^\infty\RM$ consisting of functions $f\in L^\infty\RM$ decaying at infinity in the sense that
\begin{equation*}
    \mu(\{x\in\RR : |f(x)| > \lambda\}) < \infty \quad \text{for every $\lambda > 0$}.
\end{equation*}

\begin{definition}
    We say that $(T,X,\C)$ is a \emph{compatible triple} if $\C$ is a compatible collection of quasi-Banach lattices, $X$ is a quasi-Banach lattice, and $T$ is a sublinear operator defined on $X$ having values in $\M\SN$, where $\SN$ is the common measure space for the collection $\C$.
\end{definition}

\begin{definition}
Let $(T,X,\C)$ be a compatible triple. We say that $Y\in \C$ is the \emph{optimal target for $(T,X,\C)$} if $T\colon X \to Y$ is bounded and simultaneously $Y\hookrightarrow Z$ for every $Z\in\C$ such that $T\colon X \to Z$ is bounded.
\end{definition}

A function $s\colon\RR\to\R$ will be called \emph{$\mu$-simple} (or \emph{simple} for short when the choice of the measure is clear from the context) if it can be written in the form of a finite sum \begin{equation*}
    s=\sum_{k=1}^{N}\alpha_k\chi_{E_k}, 
\end{equation*}
in which $N\in\N$ and for each $k\in\{1,\dots,N\}$, one has $\alpha_k>0$ and $E_k$ is a $\mu$-measurable subset of $\RR$ of finite measure such that
\begin{equation*}
    E_1\subseteq E_2 \subseteq \cdots \subseteq E_N\subseteq \RR.
\end{equation*} 

\begin{definition}\label{def:compatible_triple}
Let $(T,X,\C)$ be a compatible triple. We say that $T$ is \emph{simply approximable from above in $(T,X,\C)$} if $X$ contains all simple functions, for every $f\in X$ there is a sequence $\{s_j\}_{j=1}^\infty$ of simple functions  such that
\begin{equation*}
    \limsup_{j\to \infty} \|s_j\|_X \leq \|f\|_X,
\end{equation*}
and for each $Y\in\C$ there is a constant $C_Y>0$ such that
    \begin{equation*}
        \|Tf\|_{Y} \leq C_Y \limsup_{j\to\infty} \|Ts_j\|_Y.
    \end{equation*}
\end{definition}

A natural example of a compatible triple which is simply approximable from above consists of a convolution operator $T$ on $\rn$ with a nonnegative measurable kernel, a collection $\C$ of all quasi-Banach function spaces (in the sense of the definition from~\cite{NP:24}), and a quasi-Banach function space $X$. To ensure that $T$ is well defined on $X$, the convolution kernel should satisfy certain conditions depending on $X$. Instead of diving into the inherent technicalities of the general situation, let us mention a concrete example relevant to our principal question regarding \eqref{E:riesz-on-l-1}. If $X=L^{1,q}(\rn)$, $q\in(0,1]$, then a reasonable requirement is that the kernel belongs to the weak space $L^{r,\infty}(\rn)$ for some $r\in(1, \infty)$. In particular, one can take the Riesz potential $I_\gamma$ for the operator $T$, in which case the kernel belongs to $L^{\frac{n}{n-\gamma},\infty}(\rn)$.


For a function $f\in\M\RM$, we shall denote by $f_*$ its \emph{distribution function}, given by
\begin{equation*}
    f_*(\lambda)=
    \mu\left(\left\{x\in\RR:|f(x)|>\lambda\right\}\right)
    \quad\text{for $\lambda\in[0,\infty)$,}
\end{equation*}
and by $f^*$ its \emph{nonincreasing rearrangement}, defined as
\begin{equation*}
    f^*(t)=\inf\left\{\lambda\ge0:f_*(\lambda)\le t\right\}
    \quad\text{for $t\in[0,\infty)$.}
\end{equation*}
When we need to specify the measure with respect to which the rearrangement is taken, we shall write $f^*_{\mu}$, etc.

Observe that for a $\mu$-measurable subset $E$ of $\RR$, one has
\begin{equation*}
    (\chi_{E})^{*}=\chi_{[0,\mu(E))}.
\end{equation*}
It will also be useful to note that for $p,q\in(0,\infty]$, the functional governing the Lorentz space $L^{p,q}$ can be equivalently rewritten, with the help of the nonincreasing rearrangement, as
\begin{equation*}
    \|f\|_{p,q}
    =
    \|t^{\frac1p-\frac1q}f^*(t)\|_{L^q(0, \infty)},
\end{equation*}
in which $\|\cdot\|_{L^q(0, \infty)}$ stands for the ordinary Lebesgue $L^q$-(quasi-)norm.

We will say that a functional $\|\cdot\|\colon\M\RM\to[0,\infty]$ is \emph{rearrangement-invariant} if $\|f\|=\|g\|$ whenever $f^*=g^*$.

Note that every Lorentz (quasi-)norm is a rearrangement-invariant functional on $\M\RM$. 

\begin{definition}
    \label{D:fundamental-function}
    Let $\|\cdot\|_X\colon\M\RM\to[0,\infty]$ be a rearrangement-invariant functional.
    Let $X$ be the collection of all elements $f$ of $\M\RM$ such that $\|f\|_X<\infty$. Then we define the \emph{fundamental function} $\varphi_X\colon[0,\mu(\RR))\to[0,\infty]$ of $X$ by 
    \begin{equation*}
        \varphi_X(t)=\|\chi_E\|_X,
    \end{equation*}
    where $t\in[0,\mu(\RR))$ and $E$ is an arbitrary $\mu$-measurable subset of $\RR$ with $\mu(E)=t$.
\end{definition}

It can be easily verified that 
\begin{equation*}
    \varphi_{L^{p,q}\RM}(t)
    =
    c_{p,q} t^{\frac{1}{p}}
    \quad\text{for $t\in[0,\infty)$,}
\end{equation*}
for every $p\in(0,\infty]$ and 
either $q\in(0,\infty]$ if $p<\infty$ or $q=\infty$ if $p=\infty$, where $c_{p,q} = (q/p)^{1/q}$. The multiplicative constant is to be interpreted as $1$ when $q = \infty$.

\subsection{Nonimprovability of target for Lorentz endpoint domain space}

\begin{definition}
    Given a nontrivial concave function $\varphi\colon [0, \mu(\RR)) \to [0, \infty)$, we define the \emph{Lorentz endpoint space} $\Lambda_\varphi\RM$ as the collection of all functions $f\in\M\RM$ such that $\|f\|_{\Lambda_\varphi\RM}<\infty$, where 
    $\|\cdot\|_{\Lambda_\varphi\RM}\colon\M\RM\to[0,\infty]$ is the functional defined as
    \begin{equation*}
        \|f\|_{\Lambda_\varphi\RM}
        =
        \int_{0}^{\mu(\RR)}
        f^*(t)\,d\varphi(t)
    \end{equation*}
    for $f\in\M\RM$, in which the Lebesgue--Stieltjes integral is applied.
\end{definition}
The Lorentz endpoint space $\Lambda_\varphi\RM$ is a Banach lattice containing simple functions (in fact, it is a rearrangement-invariant Banach function space in the sense of \cite{BS}), and its fundamental function satisfies
\begin{equation}\label{E:Lorentz_endpoint_fund_func}
    \varphi_{\Lambda_\varphi} = \varphi.
\end{equation}

Note that in the most prominent special case when $\varphi(t)=t^{\frac{1}{p}}$ with $p\in[1,\infty)$, one has 
\begin{equation*}
    \Lambda_\varphi\RM=L^{p,1}\RM.
\end{equation*}
Here and throughout, by stating $X=Y$ for a pair of (quasi-)Banach lattices $X$ and $Y$, we mean that, as sets, they contain the same elements, and at the same time, their (quasi-)norms are equivalent in the sense that there exists a positive constant $C$ such that for all $f\in\M\RM$, one has
\begin{equation*}
    C^{-1}\|f\|_X
    \le \|f\|_Y
    \le C\|f\|_X.
\end{equation*}
We will say that a quasi-Banach lattice is \emph{normable} if it is equal to a normed space.   

\begin{theorem}\label{thm:nonimprovability_Lorentz_endpoint}
Let $\varphi\colon [0, \mu(\RR)) \to [0, \infty)$ be a nontrivial concave function. Set $X = \Lambda_\varphi\RM$. Let $(T,X,\C)$ be a compatible triple. Assume that each $Y\in\C$ is normable and that $T$ is simply approximable from above in $(T,X,\C)$. Let $W$ be a rearrangement-invariant quasi-Banach lattice over $\RM$ containing simple functions such that
\begin{equation}\label{E:nonimprovability_Lorentz_endpoint:contained_in_Lorentz}
   W\hookrightarrow X 
\end{equation}
and
\begin{equation}\label{E:nonimprovability_Lorentz_endpoint:equiv_fund_func}
    \sup_{t\in(0, \mu(\RR))} \frac{\varphi_W(t)}{\varphi(t)} < \infty.
\end{equation}
If $Y\in\C$ is the optimal target for $(T,X,\C)$, then $Y$ is also the optimal target for $(T,W,\C)$.
\end{theorem}
\begin{proof}
We start by setting
\begin{equation}\label{E:nonimprovability_Lorentz_endpoint:1}
    C_1 = \sup_{t\in(0, \mu(\RR))} \frac{\varphi_W(t)}{\varphi(t)} < \infty.
\end{equation}
By combining \eqref{E:nonimprovability_Lorentz_endpoint:contained_in_Lorentz} and the fact that $T\colon X \to Y$ is bounded, we see that $T\colon W \to Y$ is also bounded. Let $Z\in\C$ be such that $T\colon W \to Z$ is bounded. Set
\begin{equation*}
    C_2 = \|T\|_{W\to Z} < \infty.
\end{equation*}
We need to show that $Y\hookrightarrow Z$. Since $Z$ is normable, we may assume without loss of generality that $\|\cdot\|_Z$ is a norm. Let $f\in X$. Furthermore, since $T$ is simply approximable from above in $(T,X,\C)$, there are a sequence of nonnegative simple functions $\{s_j\}_{j = 1}^\infty\subseteq \MOR$ and a constant $C_Z>0$ such that
\begin{equation}\label{E:nonimprovability_Lorentz_endpoint:2}
    \limsup_{j\to\infty} \|s_j\|_{X} \leq \|f\|_X
\end{equation}
and
\begin{equation}\label{E:nonimprovability_Lorentz_endpoint:3}
    \|Tf\|_{Z} \leq C_Z \limsup_{j\to\infty} \|Ts_j\|_Z.
\end{equation}
Fix $j\in\N$. Let $s_j = \sum_{k=1}^N \alpha_k \chi_{E_k}$, where $N\in\N$, $\alpha_k>0$ for every $k\in\{1, \dots, N\}$ and $E_1\subseteq E_2 \subseteq \cdots \subseteq E_N\subseteq \RR$ are $\mu$-measurable, each of finite measure. Now, for every $k\in\{1, \dots, N\}$, using the fact that $T\colon W \to Z$ is bounded (together with the fact that $W$ contains simple functions) and \eqref{E:nonimprovability_Lorentz_endpoint:1} together with \eqref{E:Lorentz_endpoint_fund_func}, we obtain
\begin{align*}
    \|T\chi_{E_k}\|_Z \leq C_2\|\chi_{E_k}\|_W \leq C_1 C_2\|\chi_{E_k}\|_X.
\end{align*}
Hence, thanks to the sublinearity of $T$ and the lattice property of $Z$, we have
\begin{align*}
    \|Ts_j\|_Z &\leq \Big\| \sum_{k=1}^N \alpha_k T\chi_{E_k} \Big\|_Z \leq \sum_{k=1}^N \alpha_k \| T\chi_{E_k} \|_Z \\
    &\leq  C_1 C_2 \sum_{k=1}^N \alpha_k \| \chi_{E_k} \|_X = C_1 C_2 \sum_{k=1}^N  \int_0^{\mu(\RR)} \alpha_k\chi_{[0, \mu(E_k))}(t) \dd\varphi(t)  \\
    &= C_1 C_2 \int_0^{\mu(\RR)} \Big( \sum_{k=1}^N \alpha_k \chi_{E_k} \Big)^*_\mu(t) \dd\varphi(t) \\
    &= C_1 C_2 \|s_j\|_X.
\end{align*}
Combining this with \eqref{E:nonimprovability_Lorentz_endpoint:3} and \eqref{E:nonimprovability_Lorentz_endpoint:2}, we obtain
\begin{equation*}
    \|Tf\|_{Z} \leq C_1 C_2 C_Z \limsup_{j\to\infty} \|s_j\|_X \leq C_1 C_2 C_Z \|f\|_{X}.
\end{equation*}
Since $f\in X$ was arbitrary, it follows that $T\colon X \to Z$ is bounded. Finally, combining this with the fact that $Y$ is the optimal target for $(T,X,\C)$, we obtain the desired embedding $Y\hookrightarrow Z$, which concludes the proof.
\end{proof}

\begin{remark}
    The assumptions \eqref{E:nonimprovability_Lorentz_endpoint:contained_in_Lorentz} and \eqref{E:nonimprovability_Lorentz_endpoint:equiv_fund_func} imply that the fundamental function of the rearrangement-invariant quasi-Banach lattice $W$ is equivalent to that  of $X = \Lambda_\varphi$, which, in turn, is equal to $\varphi$.
\end{remark}

As the following corollary shows, we may push the conclusion of Theorem~\ref{thm:nonimprovability_Lorentz_endpoint} further by enriching the class of spaces competing for being the optimal target with some quasi-Banach lattices, which need not be normable.
\begin{corollary}\label{cor:nonimprovability_Lorentz_endpoint_and_nonenlargibility_of_possible_targets}
Let $\varphi$, $(T,X,\C)$, and $W$ be as in Theorem~\ref{thm:nonimprovability_Lorentz_endpoint}. Let $Y\in\C$ be the optimal target for $(T,X,\C)$. Assume that $\C_1$ is a collection of quasi-Banach lattices such that $\C_2 = \C \cup \C_1$ is compatible and for each $Z\in \C_1$ there is $\tilde{Z}\in \C$ such that $Z\hookrightarrow \tilde{Z} \subsetneq Y$. Then $Y$ is the optimal target for $(T,W,\C_2)$.
\end{corollary}
\begin{proof}
    By Theorem~\ref{thm:nonimprovability_Lorentz_endpoint}, we know that $Y$ is the optimal target for $(T,W,\C)$. In particular, $T\colon W\to Y$ is bounded. Let $Z\in\C_2$ be such that $T\colon W \to Z$ is bounded. We need to show that $Z\in\C$, whence it follows that $Y\hookrightarrow Z$, as desired. Suppose that $Z\in\C_1$, and let $\tilde{Z}\in \C$  be such that $Z\hookrightarrow \tilde{Z} \subsetneq Y$. Hence, $T\colon W \to \tilde{Z}$ is bounded. Since $\tilde{Z}\in \C$ and $Y$ is the optimal target for $(T,W,\C)$, it follows that $Y\hookrightarrow\tilde{Z}$. However, the combination of $Z\hookrightarrow \tilde{Z} \subsetneq Y$ and $Y\hookrightarrow\tilde{Z}$ implies that $\tilde{Z}\subsetneq \tilde{Z}$, which is clearly not possible. Hence, $Z$ must belong to $\C$, which concludes the proof.
\end{proof}

\begin{remark}
If both upper limits in the definition of the simple approximability from above are replaced by the respective lower limits, the conclusion of Theorem~\ref{thm:nonimprovability_Lorentz_endpoint} (and also of Corollary~\ref{cor:nonimprovability_Lorentz_endpoint_and_nonenlargibility_of_possible_targets}) is still valid, with only a few obvious changes needed in the proof.
\end{remark}

\begin{example}\label{exam:nonimprovability_of_riesz}
    Let $n\in\mathbb N$, $\gamma\in(0,n)$, $p\in[1,\frac{n}{\gamma}]$, and $q\in(0,1]$. Then one has
    \begin{equation}
        \label{E:riesz-on-lorentz-nonimprovable}
        I_{\gamma}\colon L^{p,q}(\rn)\to 
        \begin{cases}
            L^{\frac{np}{n-\gamma p},
            \infty}(\rn)
            &\text{if $p=1$,}
                \\
            L^{\frac{np}{n-\gamma p},1}(\rn)
            &\text{if $p\in(1,\frac{n}{\gamma})$,}
                \\
            L^{\infty}(\rn)
            &\text{if $p = \frac{n}{\gamma}$,}
            \end{cases}
    \end{equation}
    and the target space is, in each case, the optimal rearrangement-invariant Banach function space which renders~\eqref{E:riesz-on-lorentz-nonimprovable} true. In particular, the target space cannot be improved by decreasing the value of the parameter $q$ below $1$ (note that the target space is in all cases independent of $q$).
\end{example}

\begin{proof}
    The assertion is a simple consequence of Theorem~\ref{thm:nonimprovability_Lorentz_endpoint}, in which we subsequently consider $\varphi(t)=t^{\frac{1}{p}}$, $X=L^{p,1}(\rn)$, $T=I_{\gamma}$, $\mathcal C$ is the collection of all rearrangement-invariant Banach function spaces (it is important to note that we require here that these spaces are \emph{normed}), $Y$ is $L^{\frac{np}{n-\gamma p},\infty}(\rn)$, $L^{\frac{np}{n-\gamma p},1}(\rn)$, or $L^\infty(\rn)$ (depending on whether $p=1$, $p\in(1, \frac{n}{\gamma})$, or $p = \frac{n}{\gamma}$, respectively), and $W=L^{p,q}(\rn)$. 
    One only has to notice that both conditions~\eqref{E:nonimprovability_Lorentz_endpoint:contained_in_Lorentz}   and~\eqref{E:nonimprovability_Lorentz_endpoint:equiv_fund_func} are clearly satisfied. The optimality of $Y$ within the class of all rearrangement-invariant Banach function spaces was shown in \cite[Theorem~6.5]{EMMP:20}.
\end{proof}

\begin{remark}
    Example~\ref{exam:nonimprovability_of_riesz} shows how the abstract result of Theorem~\ref{thm:nonimprovability_Lorentz_endpoint} can be applied to obtain nonimprovability of the target space in a familiar situation. Interestingly, only its versions for $p=1$ and $p = \frac{n}{\gamma}$ can also be recovered by the second approach, which will be treated in Subsection~\ref{SS:calderonable} below. For the case $p\in(1, \frac{n}{\gamma})$, our second approach does not work.

    Furthermore, using Corollary~\ref{cor:nonimprovability_Lorentz_endpoint_and_nonenlargibility_of_possible_targets}, we can show that the target space in~\eqref{E:riesz-on-lorentz-nonimprovable} is in fact optimal within a larger class than that of all rearrangement-invariant Banach function spaces. Indeed, we can enlarge the collection of competing spaces with all quasi-Banach lattices $Z(\rn)$ for which there is a rearrangement-invariant Banach function space $\tilde{Z}(\rn)$ such that $Z(\rn) \hookrightarrow\tilde{Z}(\rn) \subsetneq Y_p(\rn)$, where $Y_p(\rn)$ is the target space from \eqref{E:riesz-on-lorentz-nonimprovable}.
\end{remark}

We will now present a more sophisticated (hence naturally more abstract) example, in which we shall employ the Hausdorff measure and some knowledge about traces of Sobolev functions. Resorting to the trace theory enforces us to require that $mp$ is an integer, which might seem unnatural. However, currently available methods do not seem to allow generalization to a noninteger case.

\begin{example}\label{exam:nonimprovability_of_KK_emb}
Let $1<p<n/m$, $m,n\in\N$, $m<n$, $mp\in\N$, and consider the Riesz potential $I_m$.
Let $E$ be the nonempty intersection of a $(n-mp)$-dimensional affine subspace of $\rn$ with a bounded open set in $\rn$.
Let $W$ be any rearrangement-invariant quasi-Banach lattice over $\rn$ containing simple functions such that $W\hookrightarrow L^{p,1}(\rn)$ and whose fundamental function is equivalent to the function $t\mapsto t^\frac1{p}$. Let
\begin{equation*}
    \C_1 = \{Y: Y = Y(E,\HH^{n-mp})\ \text{is an r.i.~space}\}
\end{equation*}
and
\begin{equation*}
    \C_2 = \Bigg\{Z: \begin{tabular}{@{}c@{}}
$Z = Z(E,\HH^{n-mp})$ is a quasi-Banach lattice for which exists an\\
r.i.~space $\tilde{Z}=\tilde{Z}(E,\HH^{n-mp})$ such that $Z\hookrightarrow \tilde{Z}\subsetneq L^p(E,\HH^{n-mp})$
\end{tabular} \Bigg\}.
\end{equation*}
Then the Lebesgue space $L^p(E,\HH^{n-mp})$ is the optimal target for $(I_m, W, \C)$, where $\C = \C_1 \cup \C_2$.
\end{example}

\begin{proof}
    We start by showing that the Lebesgue space $L^p(E,\HH^{n-mp})$ is the optimal target for $(I_m, L^{p,1}(\rn), \C_1)$. To this end, the boundedness of $I_m\colon L^{p,1}(\rn) \to L^p(E,\HH^{n-mp})$ follows from \cite[Theorem~1.2]{KK:18}. The optimality of $L^p(E,\HH^{n-mp})$ among all possible target spaces from $\C_1$ was essentially proved in \cite[Proposition~3.4]{CPS:20}. Indeed, let $Y=Y(E,\HH^{n-mp})$ be an r.i.~space such that
    \begin{equation}\label{E:nonimprovability_of_KK_emb:1}
        I_m\colon L^{p,1}(\rn)\to Y(E,\HH^{n-mp})\ \text{is bounded}.
    \end{equation}
    A standard potential estimate (e.g., \cite[Theorem~1.1.10/2]{Mabook}) gives the existence of a constant $C>0$, depending only on $m$ and $n$, such that
    \begin{equation}\label{E:nonimprovability_of_KK_emb:2}
        |u(x)| \leq C I_m(|\nabla^m u|)(x) \quad\text{for all $x\in\rn$ and $u\in\mathcal C^\infty_0(\rn)$},
    \end{equation}
    where $\nabla^m u$  is the (arbitrarily arranged) vector of all $m$th order derivatives of $u$. By \eqref{E:nonimprovability_of_KK_emb:1} and \eqref{E:nonimprovability_of_KK_emb:2}, we have
    \begin{equation*}
        \|u\|_{Y(E,\HH^{n-mp})} \leq \tilde{C} \|\nabla^m u\|_{L^{p,1}(\rn)} \quad \text{for every $u\in\mathcal C^\infty_0(\rn)$},
    \end{equation*}
    where $\tilde{C}$ is a constant independent of $u$. Now, since $C^\infty_0(\rn)$ is dense in the Sobolev-Lorentz space $W^m L^{p,1}(\rn)$, a standard approximation argument shows that there is a bounded linear trace operator $\operatorname{Tr}\colon W^m_0 L^{p,1}(\Omega) \to Y(E,\HH^{n-mp})$, where $\Omega\subseteq\rn$ is the open bounded set from the statement. The desired embedding $L^p(E,\HH^{n-mp})\hookrightarrow Y(E,\HH^{n-mp})$ now follows from \cite[Proposition~3.4]{CPS:20} (we are using here the assumption that $mp\in\N$).
    
    Next, it is easy to see that the assumptions of Theorem~\ref{thm:nonimprovability_Lorentz_endpoint} are satisfied with $X=L^{p,1}(\rn)$, $T=I_m$ and $\C$ there equal to $\C_1$ here. To this end, note that $L^{p,1}(\rn) = \Lambda_\varphi(\rn)$ with $\varphi(t) = t^\frac1{p}$. As for the simple approximability from above, given $f\in L^{p,1}(\rn)$, we find a nondecreasing sequence of nonnegative simple functions $\{s_j\}_{j = 1}^\infty$ such that $s_j(x)\nearrow |f(x)|$ for a.e.~$x\in\rn$. Then
    \begin{equation*}
        I_m(s_j)(x) \nearrow I_m(|f|)(x) \quad \text{for every $x\in\rn$}
    \end{equation*}
    thanks to the monotone convergence theorem. Hence, we have
    \begin{equation*}
        \|I_m f\|_{Y(E,\HH^{n-mp})} \leq \|I_m(|f|)\|_{Y(E,\HH^{n-mp})} = \lim_{j\to\infty} \|I_m(s_j)\|_{Y(E,\HH^{n-mp})}
    \end{equation*}
    for every r.i.~space $Y\in\C_1$, and
    \begin{equation*}
        \|f\|_{L^{p,1}(\rn)} = \lim_{j\to\infty}\|s_j\|_{L^{p,1}(\rn)}.
    \end{equation*}

    Finally, the desired claim follows from Corollary~\ref{cor:nonimprovability_Lorentz_endpoint_and_nonenlargibility_of_possible_targets} with $\C$, $\C_1$, and $\C_2$ there equal to $\C_1$, $\C_2$, and $\C$ here, respectively.
\end{proof}

\subsection{Endpoint mappings of Calder\'on-able operators}
\label{SS:calderonable}

Recall that a (quasi-)Banach lattice $Z$ has the \emph{weak Fatou property} if there is a constant $C>0$ such that for every function $f\in\M\RM$ and every sequence $\{f_n\}_{n=1}^\infty\subseteq Z$ such that $0\leq f_n\nearrow f$ $\mu$-a.e.~on $\RR$ and $\sup_{n\in\N}\|f_n\|_Z<\infty$, we have $f\in Z$ and $\|f\|_Z\leq C \sup_{n\in\N}\|f_n\|_Z$.

\begin{definition}
\label{def:Fatouably_repre_lattice}
We say that a rearrangement-invariant quasi-Banach lattice $X$ over $\RM$ is \emph{Fatou-representable} if there is a rearrangement-invariant quasi-Banach lattice $\widebar{X}$ over $(0, \mu(\RR))$ such that    \begin{equation}\label{E:representability}
        \|f\|_{X} = \|f^*\|_{\widebar{X}} \quad \text{for every $f\in \M\RM$}
    \end{equation}
    and $\widebar{X}$ has the weak Fatou property.
\end{definition}

Let us recall that the representability of norm-like functionals in terms of a certain symmetrized functional acting over an interval is a basic property of various rearrangement-invariant structures, and that it constitutes one of the key features in the theory. The classical version of this result is the Luxemburg representation theorem for rearrangement-invariant Banach function spaces, which first appeared in~\cite{Lux:67}
and is more easily found in~\cite[Chapter~2, Theorem~4.10]{BS}. On the other hand, its variant for rearrangement-invariant \emph{quasi}-Banach function spaces had not been known until recently, see~\cite[Theorem~3.1]{Mus:25}.
In~\cite{Mus:25}, the history of the problem is also described in more detail. An example of a rearrangement-invariant Banach lattice that is not Fatou-representable is $X = L^\infty_0\RM$ when $\mu(\RR) = \infty$. To see this, consider an increasing sequence of sets $E_1\subseteq E_2\subseteq E_3\subseteq\cdots\subseteq\RR$ such that $\bigcup_{n=1}^\infty E_n = \RR$ and $\mu(E_n) < \infty$ for every $n\in\N$. Then $\{f_n = \chi_{E_n}\}_{n = 1}^\infty \subseteq X$, $f_n^* = \chi_{[0, \mu(E_n))}$ for every $n\in\N$, and $0\leq f_n^* \nearrow \chi_{[0, \infty)}$. Now, suppose that $X$ is Fatou-representable. Using \eqref{E:representability}, we obtain
\begin{equation*}
    \sup_{n\in\N}\|f_n^*\|_{\widebar{X}} = \sup_{n\in\N}\|f_n\|_X = 1.
\end{equation*}
Since $\widebar{X}$ has the weak Fatou property, it follows that
\begin{equation}\label{E:representability_counterex}
   \chi_{[0, \infty)}\in \widebar{X} \qquad \text{and} \qquad \|\chi_{[0, \infty)}\|_{\widebar{X}} < \infty.
\end{equation}
On the other hand, since $\chi_\RR$ does not belong to $X$, we have (recall Convention~\ref{conv:qnorm_extension})
\begin{equation*}
    \|\chi_\RR\|_X = \infty.
\end{equation*}
However, using \eqref{E:representability} again together with \eqref{E:representability_counterex}, we obtain
\begin{equation*}
    \|\chi_\RR\|_X = \|\chi_{\RR}^*\|_{\widebar{X}} = \|\chi_{[0, \infty)}\|_{\widebar{X}} < \infty,
\end{equation*}
reaching a contradiction.

We will now introduce a pair of operators, albeit each of completely different nature, that will play a key role in the sequel. 

\begin{definition}
    Let $\sigma=[p,q,m]\subseteq (0, \infty)\times(0, \infty]\times(0, \infty)$. We formally define the operators $R_\sigma$ and $S^0_\sigma$ as
    \begin{equation*}
        R_\sigma g(t) = t^{-\frac1{q}} \int_0^{t^m} g(s) s^{\frac1{p} - 1} \dd{s},\ g\in\M(0, \infty),\ t\in(0, \infty),
    \end{equation*}
    and
    \begin{equation*}
        S^0_\sigma g(t) = t^{-\frac1{q}} \sup_{0<s\leq t^m} g(s) s^{\frac1{p}},\ g\in\M(0, \infty),\ t\in(0, \infty).
    \end{equation*}
\end{definition}

Employing operators acting on functions defined on an interval has been a standard practice ever since the foundations of modern theory of interpolation were laid by Lions, Peetre, Calder\'on, Zygmund, Hunt, O'Neil and others. However, the nature of these operators has undergone some evolution in accordance with the particular task in mind. Classically, such a pair of operators is intimately connected with interpolation of operators of separate \emph{weak types} or of the so-called \emph{joint weak type}. The endpoint estimates of these operators reflect their boundedness from the Lorentz space $L^{p,1}$ into the weak Lebesgue space $L^{q,\infty}$ for appropriate values of $p$ and $q$. For such endpoint estimates, the corresponding pair of operators for a Calder\'on type theorem always consists of one integral operator of \emph{Hardy type} (acting `near zero') and one integral operator of \emph{Copson type} (acting `near infinity'), for details see~\cite[Chapter~3, Definition~5.1]{BS}. 

Since then, more delicate tasks have been investigated on boundedness properties of operators with nonstandard endpoint behavior. These involve, for instance, fractional maximal operators, Sobolev embeddings, and more. Such tasks enforce introduction of operators of various special kinds. In~\cite{Gog:09}, it was observed that some nonstandard operators require a combination of an integral operator of Hardy type with a non-integral operator, which instead of integration involves the operation of supremum. The latter operator notoriously makes things more interesting owing to the simple fact that it is no longer a linear operator. More complicated operators were encountered e.g.~in~\cite{Mal:12,Bae:22,Mih:25,Kub:25}. 

Note that in the classical situation of the operators of joint weak type, the corresponding pair of operators was uniquely tight to a~\emph{segment} $\sigma$ (involving four parameters). A similar situation is encountered here, but we only need a triple of parameters, which we will also denote $\sigma$. 

Finally, we are in a position to state and prove our last two principal results.

\begin{theorem}\label{thm:nonimprovability_Calderon_endpoint_R}
    Let $\sigma=[p,q,m]\subseteq (0, \infty)\times(0, \infty]\times(0, \infty)$. Let $T_\sigma$ be either $R_\sigma$ or $S^0_\sigma$. Let $X=X\RM$ be a rearrangement-invariant quasi-Banach lattice containing $\mu$-simple functions on $\RR$. Let $T$ be a positively homogeneous operator defined on $X$ with values in $\MOS$. Assume that there are constants $C,c>0$ and a sequence of $\mu$-measurable sets $\{E_j\}_{j =1}^\infty\subseteq\RR$ such that $\lim_{j\to \infty} \mu(E_j) = 0$ and
    \begin{equation}\label{E:nonimprovability_Calderon_endpoint_R:lower_bound}
        (T\chi_{E_j})^*_\nu(t) \geq C T_\sigma(\chi_{(0,a_j)})(ct) \quad \text{for all $t\in (0, \nu(\SSS))$ and $j\in\N$},
    \end{equation}
    where $a_j = \mu(E_j)$, $j\in\N$.

    If $T\colon X \to L^{q,\infty}\SN$ and
    \begin{equation}\label{E:nonimprovability_Calderon_endpoint_R:fund}
    \limsup_{t\to0^+} t^{-\frac1{p}}\varphi_X(t) < \infty,
    \end{equation}
    then $L^{q, \infty}\SN$ is the optimal target for $(T, X, \C)$, where
    \begin{equation}\label{E:nonimprovability_Calderon_endpoint_R:class_C}
    \C = \Bigg\{Y: \begin{tabular}{@{}c@{}}
$Y = Y\SN$ is a rearrangement-invariant quasi-Banach lattice\\
that is Fatou-representable
\end{tabular} \Bigg\}.
    \end{equation}
\end{theorem}
\begin{proof}
    The multiplicative constants in this proof may depend only on the parameters $p,q$ and the constants $c, C$. Let $a=a_j$ for $j\in\N$. We have
    \begin{equation}
        R_\sigma(\chi_{(0,a)})(ct) = (ct)^{-\frac1{q}}\int_0^a s^{\frac1{p} - 1} \dd s \approx a^{\frac1{p}} t^{-\frac1{q}} \quad \text{for every $(ct)^m \geq a$}\label{E:nonimprovability_Calderon_endpoint_R:1}
    \end{equation}
    and
    \begin{equation}
         S^0_\sigma(\chi_{(0,a)})(ct) = a^{\frac1{p}} (ct)^{-\frac1{q}} \quad \text{for every $(ct)^m \geq a$}.\label{E:nonimprovability_Calderon_endpoint_R:4}
    \end{equation}
    Now, let $Y\in\C$ be such that $T\colon X \to Y$ is bounded, where $Y\in\C$. Assuming $T\colon X \to L^{q,\infty}\SN$ is bounded, we only need to show that $L^{q, \infty}\SN \hookrightarrow Y$. To this end, consider the functions
    \begin{equation*}
        f_j = a_j^{-\frac1{p}} \chi_{E_j}\in X,\ j\in\N.
    \end{equation*}
    On the one hand, using \eqref{E:nonimprovability_Calderon_endpoint_R:lower_bound}, either \eqref{E:nonimprovability_Calderon_endpoint_R:1} or \eqref{E:nonimprovability_Calderon_endpoint_R:4}, and the representability of $Y$, we have
    \begin{align}
        \|Tf_j\|_{Y} &= \|(Tf_j)_{\nu}^*\|_{\widebar{Y}} = a_j^{-\frac1{p}} \|(T\chi_{E_j})_{\nu}^*\|_{\widebar{Y}} \gtrsim a_j^{-\frac1{p}} \|T_\sigma(\chi_{(0,a_j)})(ct)\|_{\widebar{Y}} \nonumber\\
        &\gtrsim \|t^{-\frac1{q}}\chi_{(a_j^{1/m}/c,\nu(\SSS))}(t)\|_{\widebar{Y}}. \label{E:nonimprovability_Calderon_endpoint_R:2}
    \end{align}
    On the other hand, the boundedness $T\colon X \to Y$ and \eqref{E:nonimprovability_Calderon_endpoint_R:fund} imply that
    \begin{equation}\label{E:nonimprovability_Calderon_endpoint_R:3}
        \sup_{j\in\N} \|Tf_j\|_{Y} \lesssim \sup_{j\in\N} \|f_j\|_X = \sup_{j\in\N} a_j^{-\frac1{p}}\varphi_X(a_j) < \infty.
    \end{equation}
    Since $\widebar{Y}$ has the weak Fatou property, it follows from \eqref{E:nonimprovability_Calderon_endpoint_R:2} and \eqref{E:nonimprovability_Calderon_endpoint_R:3} that
    \begin{equation*}
        t^{-\frac1{q}}\chi_{(0, \nu(\SSS))}(t) \in \widebar{Y}.
    \end{equation*}
    Hence, we finally obtain
    \begin{equation*}
        \|f\|_Y = \|f^*\|_{\widebar{Y}} \leq \|t^{-\frac1{q}}\|_{\widebar{Y}} \|f\|_{L^{q, \infty}\SN}
    \end{equation*}
    for every $f\in L^{q, \infty}\SN$, that is, $L^{q, \infty}\SN \hookrightarrow Y$.
\end{proof}

\begin{remark}
    When $\nu(\SSS)<\infty$, instead of \eqref{E:nonimprovability_Calderon_endpoint_R:lower_bound}, it suffices to assume that
    \begin{equation*}
        (T\chi_{E_j})^*_\nu(t) \geq C T_\sigma(\chi_{(0,a_j)})(ct) \quad \text{for all $t\in (0, \delta)$ and $j\in\N$},
    \end{equation*}
    for some $\delta>0$ independent of $j\in\N$.
\end{remark}

\begin{example}\label{exam:nonimprovability_Calderon_endpoint_R}
    Let $\RM = \SN = \rn$ and consider the following two situations:
    \begin{enumerate}[(i)]
        \item $p=q=m=1$ and $T$ is either the Hardy--Littlewood maximal operator or Riesz/Hilbert transform. More generally, instead of the latter, $T$ may be a singular integral operator having the form
        \begin{equation*}
            Tf(x) = \lim_{\varepsilon\to 0^+} \int_{|x-y|\geq \varepsilon} \frac{\Omega(x-y)}{|x-y|^n}f(y) \dd y,\ x\in\rn,
        \end{equation*}
        where $\Omega$ is a non-identically zero odd function on $\rn\setminus\{0\}$ that is homogeneous of degree $0$ and that satisfies the following Dini-type condition:
        \begin{equation*}
            \int_0^1 \frac{\omega(\delta)}{\delta} \dd\delta < \infty,
        \end{equation*}
        where
        \begin{equation*}
            \omega(\delta) = \sup_{|x|=|y|=1, |x-y|\leq \delta} |\omega(x) - \omega(y)|,\ \delta\in(0,1).
        \end{equation*}
        Then $T$ is bounded from $L^1(\rn)$ to $L^{1,\infty}(\rn)$ and \eqref{E:nonimprovability_Calderon_endpoint_R:lower_bound} is satisfied with $T_\sigma = R_\sigma$ and $E_j = B(0,1/j)$, $j\in\N$ (see~\cite{BS,C:97,S:80}).
        \item $p=m=1$ and $q=n/(n-\alpha)$, where $\alpha\in(0,n)$. Finally, let $T$ be either the Riesz potential $I_\alpha$ or the fractional maximal operator $M_\alpha$. Then $T$ is bounded from $L^1(\rn)$ to $L^{q,\infty}(\rn)$ and \eqref{E:nonimprovability_Calderon_endpoint_R:lower_bound} is satisfied with $T_\sigma = R_\sigma$ and $E_j = B(0,1/j)$, $j\in\N$ (see~\cite{CKOP:00,S:80}).
        \end{enumerate}
        In both cases, if $X=X(\rn)$ is a rearrangement-invariant quasi-Banach lattice containing simple functions such that $X\hookrightarrow L^1(\rn)$ and such that \eqref{E:nonimprovability_Calderon_endpoint_R:fund} with $p=1$ is satisfied, then $L^{q, \infty}(\rn)$ is the optimal target for $(T, X, \C)$, where $\C$ is defined by \eqref{E:nonimprovability_Calderon_endpoint_R:class_C}.
\end{example}

\begin{definition}
    Let $\sigma=[p,q,m]\subseteq (0, \infty]\times (0, \infty] \times(0, \infty)$. We formally define the operators $H_\sigma$ and $S^\infty_\sigma$ as
    \begin{equation*}
        H_\sigma g(t) = t^{-\frac1{q}} \int_{t^m}^\infty g(s) s^{\frac1{p} - 1} \dd{s},\ g\in\M(0, \infty),\ t\in(0, \infty),
    \end{equation*}
    and
    \begin{equation*}
        S^\infty_\sigma g(t) = t^{-\frac1{q}} \sup_{t^m\leq s < \infty} g(s) s^{\frac1{p}},\ g\in\M(0, \infty),\ t\in(0, \infty).
    \end{equation*}
\end{definition}

\begin{theorem}\label{thm:nonimprovability_Calderon_endpoint_H}
    Let $\sigma=[p,q,m]\subseteq (0, \infty]\times (0, \infty] \times(0, \infty)$. Let $T_\sigma$ be either $H_\sigma$ or $S^\infty_\sigma$. Let $X=X\RM$ and $T$ be as in Theorem~\ref{thm:nonimprovability_Calderon_endpoint_R}. If $\nu(\SSS) = \infty$, assume that $\mu(\RR) = \infty$ and that there are constants $C,c>0$ and a sequence of $\mu$-measurable sets $\{E_j\}_{j =1}^\infty\subseteq\RR$ such that $\lim_{j\to \infty} \mu(E_j) = \infty$ and
    \begin{equation}\label{E:nonimprovability_Calderon_endpoint_H:lower_bound_infinite_measures}
        (T\chi_{E_j})^*_\nu(t) \geq C T_\sigma(\chi_{(0,a_j)})(ct) \quad \text{for all $t\in (0, \infty)$ and $j\in\N$},
    \end{equation}
    where $a_j = \mu(E_j)$, $j\in\N$. Otherwise, assume that there are constants $C,c>0$, a point $t_0\in(0, \nu(\SSS))$, and a $\mu$-measurable set $E\subseteq \RR$ with $0<a=\mu(E)\leq \mu(\RR)$ such that
    \begin{equation}\label{E:nonimprovability_Calderon_endpoint_H:lower_bound_finite_measure_range}
        (T\chi_E)^*_\nu(t) \geq C T_\sigma(\chi_{(0,a)})(ct) \quad \text{for every $t\in (0, t_0)$}.
    \end{equation}

    If $T\colon X \to L^{q,\infty}\SN$ is bounded and
   \begin{equation}\label{E:nonimprovability_Calderon_endpoint_H:fund_infinite_measures}
    \limsup_{t\to\infty} t^{-\frac1{p}}\varphi_X(t) < \infty \quad \text{when $\mu(\RR) = \nu(\SSS) = \infty$},
    \end{equation}
    then $L^{q, \infty}\SN$ is the optimal target for $(T, X, \C)$, where $\C$ is defined by \eqref{E:nonimprovability_Calderon_endpoint_R:class_C}.
\end{theorem}
\begin{proof}
    The multiplicative constants in this proof may depend only on the parameters $p,q$ and the constants $c, C$. Let $b>0$. When $p<\infty$, we have
    \begin{equation}\label{E:nonimprovability_Calderon_endpoint_H:1}
        H_\sigma(\chi_{(0,b)})(ct) = (ct)^{-\frac1{q}}\int_{(ct)^m}^b s^{\frac1{p} - 1} \dd s \approx b^{\frac1{p}} t^{-\frac1{q}} \quad \text{for every $(ct)^m \leq b/2$}. 
    \end{equation}
    When $p = \infty$, we have
    \begin{align}
        H_\sigma(\chi_{(0,b)})(ct) &= (ct)^{-\frac1{q}}\int_{(ct)^m}^b s^{-1} \dd s \nonumber\\
        &\gtrsim t^{-\frac1{q}}\int_{b/2}^b s^{-1} \dd s \approx t^{-\frac1{q}} \quad \text{for every $(ct)^m \leq b/2$}.
    \end{align}
    Furthermore, in either case, we have
    \begin{equation}\label{E:nonimprovability_Calderon_endpoint_H:9}
        S^\infty_\sigma(\chi_{(0,b)})(ct) \approx t^{-\frac1{q}} b^{\frac1{p}} \quad \text{for every $(ct)^m \leq b/2$}.
    \end{equation}
    Now, let $Y\in\C$ be such that $T\colon X \to Y$ is bounded. Assume that $T\colon X \to L^{q,\infty}\SN$ is bounded. First, assume that $\mu(\RR) = \nu(\SSS) = \infty$, and set
    \begin{equation*}
        f_j = a_j^{-\frac1{p}} \chi_{E_j}\in X,\ j\in\N.
    \end{equation*}
    On the one hand, using \eqref{E:nonimprovability_Calderon_endpoint_H:lower_bound_infinite_measures}, one of \eqref{E:nonimprovability_Calderon_endpoint_H:1}--\eqref{E:nonimprovability_Calderon_endpoint_H:9}, and the representability of $Y$, we have
    \begin{align}
        \|Tf_j\|_{Y} &= \|(Tf_j)^*\|_{\widebar{Y}} = a_j^{-\frac1{p}} \|(T\chi_{E_j})^*\|_{\widebar{Y}} \gtrsim a_j^{-\frac1{p}} \|T_\sigma(\chi_{(0,a_j)})(ct)\|_{\widebar{Y}} \nonumber\\
        &\gtrsim \|t^{-\frac1{q}}\chi_{(0,(a_j/2)^{1/m}/c)}(t)\|_{\widebar{Y}}\label{E:nonimprovability_Calderon_endpoint_H:3}
    \end{align}
    for every $j\in\N$.  On the other hand, the boundedness $T\colon X \to Y$ and \eqref{E:nonimprovability_Calderon_endpoint_H:fund_infinite_measures} imply that
    \begin{equation}\label{E:nonimprovability_Calderon_endpoint_H:4}
        \sup_{j\in\N} \|Tf_j\|_{Y} \lesssim \sup_{j\in\N} \|f_j\|_X = \sup_{j\in\N} a_j^{-\frac1{p}}\varphi_X(a_j) < \infty.
    \end{equation}
    Since $\widebar{Y}$ has the weak Fatou property, it follows from \eqref{E:nonimprovability_Calderon_endpoint_H:3} and \eqref{E:nonimprovability_Calderon_endpoint_H:4} that
    \begin{equation}\label{E:nonimprovability_Calderon_endpoint_H:8}
        t^{-\frac1{q}} \in \widebar{Y}.
    \end{equation}
    Hence, we finally obtain
    \begin{equation*}
        \|f\|_Y = \|f^*\|_{\widebar{Y}} \leq \|t^{-\frac1{q}}\|_{\widebar{Y}} \|f\|_{L^{q, \infty}\SN}
    \end{equation*}
    for every $f\in L^{q, \infty}\SN$, that is, $L^{q, \infty}\SN \hookrightarrow Y$. When $\nu(\SSS)<
    \infty$, we set
    \begin{equation*}
        f = a^{-\frac1{p}}\chi_{E}\in X
    \end{equation*}
    and use \eqref{E:nonimprovability_Calderon_endpoint_H:lower_bound_finite_measure_range} instead of \eqref{E:nonimprovability_Calderon_endpoint_H:lower_bound_infinite_measures}. Similarly to \eqref{E:nonimprovability_Calderon_endpoint_H:3}, it follows that
    \begin{equation*}
        \|Tf\|_{Y} \gtrsim \|t^{-\frac1{q}}\chi_{(0,a_0)}(t)\|_{\widebar{Y}},
    \end{equation*}
    where $a_0 = \min\{t_0, (a/2)^{1/m}/c\}$. Hence
    \begin{equation}\label{E:nonimprovability_Calderon_endpoint_H:6}
        \|t^{-\frac1{q}}\chi_{(0,a_0)}(t)\|_{\widebar{Y}} \lesssim \|Tf\|_{Y} \lesssim \|f\|_X = a^{-\frac1{p}} \varphi_X(a) < \infty.
    \end{equation}
    Furthermore, using the fact that $\nu(\SSS) < \infty$, we also have
    \begin{equation}\label{E:nonimprovability_Calderon_endpoint_H:7}
        \|t^{-\frac1{q}}\chi_{(a_0,\nu(\SSS))}(t)\|_{\widebar{Y}} \leq (a_0)^{-\frac1{q}} \|\chi_{\SSS}\|_{\widebar{Y}} < \infty.
    \end{equation}
    Combining \eqref{E:nonimprovability_Calderon_endpoint_H:6} and \eqref{E:nonimprovability_Calderon_endpoint_H:7}, we obtain \eqref{E:nonimprovability_Calderon_endpoint_H:8}, whence the desired embedding $L^{q, \infty}\SN \hookrightarrow Y$ follows as before.
\end{proof}

\begin{example}
     Let $\RM = \SN = \rn$, $\alpha\in(0, n)$, $p = n/\alpha$, $q = \infty$, and $m = 1$. For $T = I_\alpha$ or $T = M_\alpha$, \eqref{E:nonimprovability_Calderon_endpoint_H:lower_bound_infinite_measures} is satisfied with $T_\sigma = H_\sigma$ or $T_\sigma = S_\sigma^\infty$, respectively, and $E_j = B(0,j)$, $j\in\N$ (see~\cite{CKOP:00,S:80}). Assume that $X=X(\rn)$ is a rearrangement-invariant quasi-Banach lattice containing simple functions such that \eqref{E:nonimprovability_Calderon_endpoint_H:fund_infinite_measures} is satisfied and $X\hookrightarrow L^{n/\alpha,1}(\rn)$ when $T = I_\alpha$; $X\hookrightarrow L^{n/\alpha,\infty}(\rn)$ when $T = M_\alpha$. Then $L^\infty(\rn)$ is the optimal target for $(T, X, \C)$, where $\C$ is defined by \eqref{E:nonimprovability_Calderon_endpoint_R:class_C}.
\end{example}

\section*{Acknowledgment}
The authors would like to thank the referee for reading the paper and their valuable comments.



\begin{thebibliography}{29}
\providecommand{\natexlab}[1]{#1}
\providecommand{\url}[1]{\texttt{#1}}
\expandafter\ifx\csname urlstyle\endcsname\relax
  \providecommand{\doi}[1]{doi: #1}\else
  \providecommand{\doi}{doi: \begingroup \urlstyle{rm}\Url}\fi

\bibitem[Adams and Hedberg(1996)]{AH:96}
D.R. Adams and L.I. Hedberg.
\newblock \emph{Function spaces and potential theory}, volume 314 of
  \emph{Grundlehren der Mathematischen Wissenschaften [Fundamental Principles
  of Mathematical Sciences]}.
\newblock Springer-Verlag, Berlin, 1996.
\newblock \doi{10.1007/978-3-662-03282-4}.

\bibitem[Baena-Miret et~al.(2022)Baena-Miret, Gogatishvili, Mihula, and
  Pick]{Bae:22}
S.~Baena-Miret, A.~Gogatishvili, Z.~Mihula, and L.~Pick.
\newblock Reduction principle for {G}aussian {$K$}-inequality.
\newblock \emph{J. Math. Anal. Appl.}, 516\penalty0 (2):\penalty0 Paper No.
  126522, 23 pages, 2022.
\newblock \doi{10.1016/j.jmaa.2022.126522}.

\bibitem[Bennett and Rudnick(1980)]{BR:80}
C.~Bennett and K.~Rudnick.
\newblock On {L}orentz-{Z}ygmund spaces.
\newblock \emph{Dissertationes Math. (Rozprawy Mat.)}, 175:\penalty0 1--67,
  1980.

\bibitem[Bennett and Sharpley(1988)]{BS}
C.~Bennett and R.~Sharpley.
\newblock \emph{{Interpolation of operators}}, volume 129 of \emph{Pure and
  Applied Mathematics}.
\newblock Academic Press, Inc., Boston, MA, 1988.

\bibitem[Calder\'on(1966)]{C:66}
A.-P. Calder\'on.
\newblock Spaces between {$L\sp{1}$} and {$L\sp{\infty }$} and the theorem of
  {M}arcinkiewicz.
\newblock \emph{Studia Math.}, 26:\penalty0 273--299, 1966.
\newblock \doi{10.4064/sm-26-3-301-304}.

\bibitem[Cianchi(1997)]{C:97}
A.~Cianchi.
\newblock A note on two-weight inequalities for maximal functions and singular
  integrals.
\newblock \emph{Bull. London Math. Soc.}, 29\penalty0 (1):\penalty0 53--59,
  1997.
\newblock \doi{10.1112/S0024609396001798}.

\bibitem[Cianchi et~al.(2000)Cianchi, Kerman, Opic, and Pick]{CKOP:00}
A.~Cianchi, R.~Kerman, B.~Opic, and L.~Pick.
\newblock A sharp rearrangement inequality for the fractional maximal operator.
\newblock \emph{Studia Math.}, 138\penalty0 (3):\penalty0 277--284, 2000.
\newblock \doi{10.4064/sm-138-3-277-284}.

\bibitem[Cianchi et~al.(2015)Cianchi, Pick, and Slav\'{\i}kov\'{a}]{CPS:15}
A.~Cianchi, L.~Pick, and L.~Slav\'{\i}kov\'{a}.
\newblock Higher-order {S}obolev embeddings and isoperimetric inequalities.
\newblock \emph{Adv. Math.}, 273:\penalty0 568--650, 2015.
\newblock \doi{10.1016/j.aim.2014.12.027}.

\bibitem[Cianchi et~al.(2020)Cianchi, Pick, and Slav\'{\i}kov\'{a}]{CPS:20}
A.~Cianchi, L.~Pick, and L.~Slav\'{\i}kov\'{a}.
\newblock Sobolev embeddings, rearrangement-invariant spaces and {F}rostman
  measures.
\newblock \emph{Ann. Inst. H. Poincar\'{e} C Anal. Non Lin\'{e}aire},
  37\penalty0 (1):\penalty0 105--144, 2020.
\newblock \doi{10.1016/j.anihpc.2019.06.004}.

\bibitem[Edmunds et~al.(2000)Edmunds, Kerman, and Pick]{EKP:00}
D.E. Edmunds, R.~Kerman, and L.~Pick.
\newblock Optimal {S}obolev imbeddings involving rearrangement-invariant
  quasinorms.
\newblock \emph{J. Funct. Anal.}, 170\penalty0 (2):\penalty0 307--355, 2000.
\newblock \doi{10.1006/jfan.1999.3508}.

\bibitem[Edmunds et~al.(2020)Edmunds, Mihula, Musil, and Pick]{EMMP:20}
D.E. Edmunds, Z.~Mihula, V.~Musil, and L.~Pick.
\newblock Boundedness of classical operators on rearrangement-invariant spaces.
\newblock \emph{J. Funct. Anal.}, 278\penalty0 (4):\penalty0 108341, 56 pages,
  2020.
\newblock \doi{10.1016/j.jfa.2019.108341}.

\bibitem[Gogatishvili and Pick(2009)]{Gog:09}
A.~Gogatishvili and L.~Pick.
\newblock Calder\'on-type theorems for operators of nonstandard endpoint
  behaviour.
\newblock \emph{Indiana Univ. Math. J.}, 58\penalty0 (4):\penalty0 1831--1851,
  2009.
\newblock \doi{10.1512/iumj.2009.58.3636}.

\bibitem[Hernandez et~al.(2023)Hernandez, Rai\c{t}\u{a}, and Spector]{HRS23}
F.~Hernandez, B.~Rai\c{t}\u{a}, and D.~Spector.
\newblock Endpoint {$L^1$} estimates for {H}odge systems.
\newblock \emph{Math. Ann.}, 385\penalty0 (3-4):\penalty0 1923--1946, 2023.
\newblock \doi{10.1007/s00208-022-02383-y}.

\bibitem[Korobkov and Kristensen(2018)]{KK:18}
M.V. Korobkov and J.~Kristensen.
\newblock The trace theorem, the {L}uzin {$N$}- and {M}orse-{S}ard properties
  for the sharp case of {S}obolev-{L}orentz mappings.
\newblock \emph{J. Geom. Anal.}, 28\penalty0 (3):\penalty0 2834--2856, 2018.
\newblock \doi{10.1007/s12220-017-9936-7}.

\bibitem[Kub\'{\i}\v{c}ek()]{Kub:25}
D.~Kub\'{\i}\v{c}ek.
\newblock Nonstandard {C}alder\'{o}n-type theorems.
\newblock Preprint.\ arXiv:2512.06826 [math.FA].

\bibitem[Lieb and Loss(2001)]{LL:01}
E.H. Lieb and M.~Loss.
\newblock \emph{Analysis}, volume~14 of \emph{Graduate Studies in Mathematics}.
\newblock American Mathematical Society, Providence, RI, second edition, 2001.
\newblock \doi{10.1090/gsm/014}.

\bibitem[Luxemburg(1967)]{Lux:67}
W.A.J. Luxemburg.
\newblock Rearrangement-invariant {B}anach function spaces.
\newblock In \emph{Proc. Sympos. in Analysis}, volume~10 of \emph{Queen’s
  Papers in Pure and Appl. Math.}, pages 83--144, 1967.

\bibitem[Mal\'y and Pick(2002)]{Mal:02}
J.~Mal\'y and L.~Pick.
\newblock An elementary proof of sharp {S}obolev embeddings.
\newblock \emph{Proc. Amer. Math. Soc.}, 130\penalty0 (2):\penalty0 555--563,
  2002.
\newblock \doi{10.1090/S0002-9939-01-06060-9}.

\bibitem[Mal\'{y}(2012)]{Mal:12}
L.~Mal\'{y}.
\newblock Calder\'on-type theorems for operators with non-standard endpoint
  behavior on {L}orentz spaces.
\newblock \emph{Math. Nachr.}, 285\penalty0 (11-12):\penalty0 1450--1465, 2012.
\newblock \doi{10.1002/mana.201100095}.

\bibitem[Maz'ya(2011)]{Mabook}
V.G. Maz'ya.
\newblock \emph{Sobolev spaces with applications to elliptic partial
  differential equations}, volume 342 of \emph{Grundlehren der Mathematischen
  Wissenschaften [Fundamental Principles of Mathematical Sciences]}.
\newblock Springer, Heidelberg, augmented edition, 2011.
\newblock \doi{10.1007/978-3-642-15564-2}.

\bibitem[Mihula et~al.(2026)Mihula, Pick, and Spector]{Mih:25}
Z.~Mihula, L.~Pick, and D.~Spector.
\newblock Potential trace inequalities via a {C}alder\'on-type theorem.
\newblock \emph{J. Lond. Math. Soc. (2)}, 113\penalty0 (3):\penalty0 Paper No.
  e70504, 35 pages, 2026.
\newblock \doi{10.1112/jlms.70504}.

\bibitem[Musilov\'a et~al.(2025)Musilov\'a, Nekvinda, Pe\v{s}a, and
  Tur\v{c}inov\'a]{Mus:25}
A.~Musilov\'a, A.~Nekvinda, D.~Pe\v{s}a, and H.~Tur\v{c}inov\'a.
\newblock On the properties of rearrangement-invariant quasi-{B}anach function
  spaces.
\newblock \emph{Nonlinear Anal.}, 260:\penalty0 Paper No. 113854, 30 pages,
  2025.
\newblock \doi{10.1016/j.na.2025.113854}.

\bibitem[Nekvinda and Pe\v{s}a(2024)]{NP:24}
A.~Nekvinda and D.~Pe\v{s}a.
\newblock On the properties of quasi-{B}anach function spaces.
\newblock \emph{J. Geom. Anal.}, 34\penalty0 (8):\penalty0 Paper No. 231, 29
  pages, 2024.
\newblock \doi{10.1007/s12220-024-01673-y}.

\bibitem[Sadosky(1979)]{Sad:79}
C.~Sadosky.
\newblock \emph{Interpolation of operators and singular integrals}, volume~53
  of \emph{Monographs and Textbooks in Pure and Applied Mathematics}.
\newblock Marcel Dekker, Inc., New York, 1979.
\newblock An introduction to harmonic analysis.

\bibitem[Schikorra et~al.(2017)Schikorra, Spector, and Van~Schaftingen]{SSVS17}
A.~Schikorra, D.~Spector, and J.~Van~Schaftingen.
\newblock An {$L^1$}-type estimate for {R}iesz potentials.
\newblock \emph{Rev. Mat. Iberoam.}, 33\penalty0 (1):\penalty0 291--303, 2017.
\newblock \doi{10.4171/RMI/937}.

\bibitem[Sharpley(1980)]{S:80}
R.~Sharpley.
\newblock Counterexamples for classical operators on {L}orentz-{Z}ygmund
  spaces.
\newblock \emph{Studia Math.}, 68\penalty0 (2):\penalty0 141--158, 1980.
\newblock \doi{10.4064/sm-68-2-141-158}.

\bibitem[Spector(2020)]{Spector20}
D.~Spector.
\newblock An optimal {S}obolev embedding for {$L^1$}.
\newblock \emph{J. Funct. Anal.}, 279\penalty0 (3):\penalty0 108559, 26 pages,
  2020.
\newblock \doi{10.1016/j.jfa.2020.108559}.

\bibitem[Stolyarov(2022)]{Stolyarov22}
D.M. Stolyarov.
\newblock Hardy-{L}ittlewood-{S}obolev inequality for {$p=1$}.
\newblock \emph{Mat. Sb.}, 213\penalty0 (6):\penalty0 125--174, 2022.
\newblock \doi{10.4213/sm9645}.

\bibitem[Ziemer(1989)]{Z:89}
W.P. Ziemer.
\newblock \emph{Weakly differentiable functions}, volume 120 of \emph{Graduate
  Texts in Mathematics}.
\newblock Springer-Verlag, New York, 1989.
\newblock \doi{10.1007/978-1-4612-1015-3}.
\newblock Sobolev spaces and functions of bounded variation.

\end{thebibliography}
\end{document}